\documentclass[]{interact}

\usepackage{epstopdf}
\usepackage[caption=false]{subfig}
\usepackage{setspace}
\usepackage{subfig}
\usepackage{indentfirst}
\usepackage{graphicx}
\usepackage{xcolor}
\usepackage{fancyhdr}
\usepackage{url}
\usepackage{enumerate}
\usepackage{amsmath, amsthm, amsfonts, amssymb, amsxtra}
\usepackage[square,comma,numbers,sort&compress]{natbib}
\usepackage{bm}
\usepackage{subfig}
\usepackage{float}

\usepackage{booktabs}
\usepackage{lscape}
\usepackage[numbers,sort&compress]{natbib}
\bibpunct[, ]{[}{]}{,}{n}{,}{,}
\renewcommand\bibfont{\fontsize{10}{12}\selectfont}

\theoremstyle{plain}
\newtheorem{theorem}{Theorem}[section]

\theoremstyle{definition}
\newtheorem{definition}[theorem]{Definition}

\theoremstyle{remark}

\newcommand{\f}{\operatorname}

\begin{document}

\articletype{ARTICLE TEMPLATE}

\title{On the unification of zero-adjusted cure survival models}

\author{
\name{Francisco Louzada\textsuperscript{a}, Pedro L. Ramos\textsuperscript{b}, Hayala C. C. de Souza\textsuperscript{c}\thanks{CONTACT Pedro Ramos N. Email: pedro.ramos@mat.uc.cl}, Lawal Oyeneyin\textsuperscript{c} and , Gleici da S. C. Perdon\'a\textsuperscript{d}}
\affil{
\textsuperscript{a}Institute of Mathematical and Computer Sciences,  S\~ao Paulo University, S\~ao Carlos, Brazil; \\ \textsuperscript{b}Pontificia Universidad Católica de Chile, Santiago, Chile;\\
\textsuperscript{c} Department of Obstetrics and Gynaecology, Mother and Child Hospital, Ondo State, Nigeria\\
\textsuperscript{d} Department of Social Medicine - Ribeir\~ao Preto Medical School, University of S\~ao Paulo, Ribeir\~ao Preto, Brazil 
}
}

\maketitle

\begin{abstract}

This paper proposes a unified version of survival models that accounts for both zero-adjustment and cure proportions in various latent competing causes, useful in data where survival times may be zero or cure proportions are present. These models are particularly relevant in scenarios like childbirth duration in sub-Saharan Africa. Different competing cause distributions were considered, including Binomial, Geometric, Poisson, and Negative Binomial. The model's maximum likelihood point estimators and asymptotic confidence intervals were evaluated through simulation, demonstrating improved accuracy with larger sample sizes. The model best fits real obstetric data when assuming geometrically distributed causes. This flexible model, capable of considering different distributions for the lifetime of susceptible individuals and competing causes, is an effective tool for adjusting survival data, indicating broad application potential.
\end{abstract}


\section{Introduction}
\textcolor{black}{Survival models that incorporate a long-term (or cure) proportion, capturing the subset of individuals who do not experience the event of interest even after extended follow-up, have become standard practice. There are several examples of applying this class of models in the biomedical field \cite{perdona2011general,martinez2014bayesian,martinez2013mixture,mazucheli2013exponentiated, cobre2013mechanistic}. 
The most popular cure models are the mixture cure model proposed by Berkson and Gage \cite{berkson1952survival} and the promotion time cure rate model introduced by Yakovlev et al. \cite{tsodikov1996stochastic} and generalized by \cite{yin2005general,cooner2007flexible}. Their main motivation was related to the survival of cancer patients or even the development of oncogene cells.  Perhaps, trying to find a single answer, Tsodikov,Ibrahim, and Yakovlev \cite{tsodikov2003estimating}, proposed a unified approach for cure survival models that was later revisited by Cooner, et.al \cite{cooner2007flexible} and by Rodrigues et al. \cite{rodrigues2009unification}. 
Specifically, the models examined the relationship between cell numbers (causes) related to cancer survival, with the number of concurrent causes or cells (N) treated as a random variable. All proposed models assume that the variable time must be greater than zero. }

Although survival models typically assume that lifetimes are always greater than zero, there are numerous phenomena where outcomes can include zero lifetimes, affecting the survival curve by deflating its initial value to values smaller than one.
 For example, in the commercial area, where the interest is to predict defaulters or fraudsters, a proportion of clients do not pay their loans from the beginning of the contract, leading to default times equal to zero \cite{louzada2018zero}. 
In the field of oncology, there is a significant focus on examining the time interval that elapses between a cancer diagnosis and the onset of metastasis,it is possible that patients already present metastasis at the time of diagnosis \cite{marko2016prevalence}. 

 Also, in oncology, when the researcher is interested in evaluating the lifetime after surgical treatment, a proportion of individuals may die during or immediately after surgery, generating lifetimes equal to zero \cite{brenna2004prognostic}. This peculiarity is the so-called zero-adjustment or zero-inflation  \cite{kalamatianou2003perpetual,de2017zero,louzada2018zero}.
Additionally, survival models that accommodate these two characteristics—zero adjustment and cure ratio—play a key role in survival analysis and are referred to as zero-adjusted cure (ZAC) survival models \cite{louzada2018zero, de2017zero, calsavara2019zero, cavenague2022log, de2022bayesian}. Such models can be compared with the distribution of a latent variable, $N$, which represents the count of possible competing causes for the occurrence of the outcome of interest.

The first is the standard ZAC mixture model proposed by Kalamatianou and McClean \cite{kalamatianou2003perpetual} and further formalized by Louzada et al. \cite{louzada2018zero}, which generalizes the mixture cure model proposed by Berkson and Gage \cite{berkson1952survival}. Balakrishnan and Milienos \cite{balakrishnan2020class} provide an in-depth analysis of zero-adjusted cure models, along with further extensions \cite{milienos2022reparameterization, papastamoulis2024bayesian}. In this model, the survival function of the population is given by $S(t) = p_0 + (1 - p_0 - p_1)S_0(t)$, where $p_0$ is the cure proportion, $p_1$ is the zero-adjustment proportion, and $S_0(t)$ is the survival function of the susceptible population. It is possible to show that in the standard mixture ZAC model, the number of competing causes, $N$, follows a Bernoulli distribution. 

The second one is the promotion ZAC model \cite{de2017zero}, which generalizes the promotion cure model \cite{chen1999new}, assuming that $N$ follows a Poisson distribution with mean equal to $\theta$. Then, the population survival function is given by
$$
S(t) = e^{-\theta} + \left( 1 - e^{-\theta} - e^{-\kappa} \right)\left[\frac{e^{-\theta F_0(t)} - e^{-\theta}}{1 - e^{-\theta}} \right].
$$
Therefore, the cure and zero-adjusted proportions are given by $e^{-\theta}$ and $e^{-\kappa}$, respectively. A third model \cite{calsavara2019zero} where the population survival function is given by $S(t) = (1-p)S_0(t)$, which $p$ denotes the proportion of zero-adjusted lifetime, and $S_0(t)$ is an improper survival function from Gompertz or Inverse Gaussian defective models.  In light of what was described, the proposed models are restricted regarding the distribution of $N$. Considering more flexible distributions for $N$ can be useful to describe intricate patterns of data and identify which distribution best fits the data. Therefore, the ZAC models must be more flexible in modeling data with competing risks.

In this paper, to allow the distribution of $N$ to be more flexible than Poisson or Bernoulli distributions, we propose a unified version of the survival models, hereafter the unified ZAC model, which accommodates the zero-adjustment and cure proportion for a general class of latent competing causes. Our modeling is motivated by the peculiarities observed in data from a sub-Saharan African obstetric study. This data set was collected by the World Health Organization (WHO) as part of the development of the project Better Outcomes in Labour Difficulty (BOLD) \cite{oladapo2015better,souza2015development,oladapo2018progression}. 
More details are provided in the application section.

The paper is organized as follows. In Section 2, we derive the unified zero-adjusted cure survival model and present some particular cases as well as, we discuss the inferential procedures under the maximum likelihood approach, assuming that the $N$ follows a Negative Binomial distribution. In Section 3, we present the results of a simulation study to evaluate the finite sample performance of parameter estimates of our model. Section 4 illustrates the proposed methodology on a sub-Saharan African obstetric dataset. The discussion and final remarks are presented in Section 5.

\section{Methods}
\subsection{The unified zero-adjusted cure survival model}\label{sec2}

Following Rodrigues et al. \cite{rodrigues2009unification}, let $Z_i \ (i=1,2,.\ldots)$ be defined as the time-to-event due to the $i$-th competing risk. Also, let $N$ be a random variable with a discrete distribution  $p_n=P(N=n), n=0,1,2,\ldots$, where $Z_i$ are independent of $N$. In order to account the individuals that are not susceptible to the event of interest, the lifetime is defined as $Y=\min(Z_1,Z_2,\cdots, Z_N)$, with $P[Z_0=\infty]=1$, where $Z_0$ is conditionally independent of $N$, leading to a proportion $p_0$ of the population which is not susceptible to the event of interest. Additionally, according to  Feller \cite{feller}, let ${a_n}$ be a real number sequence. 
It is known that the function $A(s)= a_0+a_1 s+a_2 s^2+\cdots$ can be defined as the generating function of the sequence $\{a_n\}$ if, for $s\in[0,1]$,  $A(s)$ converges.

The standard survival models suppose that $Z_i, i=1,\ldots,n$, is a non-negative random variable. Then, to present a model that takes into account zero initial failures, let us define

\begin{equation*}
\begin{aligned}
W_i=
\begin{cases}
0 & \text{
if the observed failure time is equal to 0
}, \\ 
Z_i, & \text{otherwise },
\end{cases}
\end{aligned}
\end{equation*}
as realizations of W. 
To account the individuals that are not susceptible to the event of \textcolor{black}{interest,
the lifetime}, $Y$, is redefined as $Y=\min(W_1,W_2,\cdots, W_N)$, where $P[W_0=\infty]=1$.    

\begin{definition}\label{teor3}  If for $S(y)\in[0,1]$, then
\begin{equation}
A^*(S(y))= a_0+ (a_{1}-b_{1})S(y)+(a_{2}-b_{2})S(y)^2+\cdots= A(S(y))- \sum_{n=1}^{\infty}b_n(S(y)^n),
\end{equation}

\noindent where $A(\cdot)$ and $\sum_{n=1}^{\infty}b_n(S(y)^n)$ converges if $0 \leq S(y) \leq 1$.
\end{definition} 

Based on Definition  \ref{teor3} and on the definitions of Feller\cite{feller} and Rodrigues et al.\cite{rodrigues2009unification}, the survival function of the unified ZAC model, $S_{zp}(y)$, is defined in the following Theorem.

\begin{theorem} \textcolor{black}{Let ${a_n}$and ${b_n}$ be a sequence of real numbers}. Given a proper survival function, S(y), the survival function of the random variable Y is given by
\begin{equation} \label{szp}
\begin{aligned}
S_{zp}(y)&=A^*(S(y))=\sum_{n=0}^{\infty} a_n\left\{ S(y)\right\}^n - \sum_{n=1}^{\infty}b_n\left\{ S(y)\right\}^n\\&=A(S(y))- \sum_{n=1}^{\infty}b_n(S(y)^n).
\end{aligned}
\end{equation}
where $A(.)$ and $\sum_{n=1}^{\infty}b_n(S(y)^n)$ converges if $0 \leq S(y) \leq 1$.
\end{theorem}

\begin{proof} Firstly, we would like to recall the proof of Theorem 2 presented in Rodrigues et al. which shows that if $Z_k$ denote the time-to-event due to the competing cause $k$ and is greater than 0 then
\begin{equation}\label{proof11}
\begin{aligned}
S_{p}(y)&=P[N=0]+P(Z_1>y,\ldots,Z_N>y,N\geq1)\\
& = a_0 + \sum_{n=1}^{\infty} a_n\left\{ S(y)\right\}^n =A(S(y)).
\end{aligned}
\end{equation}
which is the unified long-term survival model. Note that, here the authors are assuming that $y>0$. Let us suppose that $Z_k$ denote the time-to-event due to the competing cause $k$ and can assume 0 value for at least one $k$.  It is important to point out that, in this case the survival function is defined as $S_p(y)=P[Y\geq y]$ since $P[Y=y]\neq0$ if $p_1\neq0$. Hence, we have that
\begin{equation}\label{proof12}
\begin{aligned}
P(Z_1\geq y,\ldots,Z_N\geq y,N\geq1)=&\ P(Z_1> y,\ldots,Z_N> y,N\geq1)\\&+P\left(\mbox{at least one } Z_i=y ,N\geq1\right).
\end{aligned}
\end{equation}

From (\ref{proof11}) we have that
\begin{equation*}
\begin{aligned}
S_{zp}(y)&=P[N=0]+P(Z_1>y,\ldots,Z_N>y,N\geq1)  \\
&	\overset{(4)}{=}P[N=0]+P(Z_1\geq y,\ldots,Z_N\geq y,N\geq1)-P\left(\mbox{at least one } Z_i=y ,N\geq1\right)  \\
& = a_0 + \sum_{n=1}^{\infty} a_n\left\{ S(y)\right\}^n -\sum_{n=1}^{\infty}b_n\left\{ S(y)\right\}^n\\&=A(S(y))- \sum_{n=1}^{\infty}b_n(S(y)^n).
\end{aligned}
\end{equation*}
\end{proof}

Note that when $N=0$, we have that 
$P\left(\mbox{at least one } Z_i=0, N\geq1\right)=0$ and $P(Z_1>y,\ldots,Z_N>y,N\geq1)=0$,
and $S_{zp}(y)=P[N=0]=a_0$. 
On the other hand, if $N\geq1$, and we observe $y=0$, where $Y=\min(Z_1,Z_2$, $\ldots,Z_N)$, then at least one $Z_i=0$.
Hence, we have that 
$$P\left(\mbox{at least one } Z_i=0 ,N\geq1\right)=\sum_{n=1}^{\infty}b_n=p_1. $$

Additionally,
$\lim_{y \to 0} P(Z_1>y,\ldots,Z_N>y,N\geq1) = 1.$
Therefore, 
$$\lim_{y \to 0}S_{zp}(y)=1-p_1.$$

The unified ZAC model (\ref{szp}) generalizes various survival models with and without zero adjustment. This particularity is demonstrated in the following two Theorems of the ZAC models proposed by \textcolor{black}{Louzada et al.\cite{louzada2018zero}, and Oliveira et al.\cite{de2017zero}.}
\vspace{1cm}

\begin{theorem} \label{teor5}
The Standard Mixture ZAC  
model is given by
\begin{equation}\label{zcuremodel}
S_{zp}(y)=p_0+(1-p_1-p_0)S_p^*(y),
\end{equation}
where
\begin{equation*}
S_p^*(y)=\sum_{n=1}^{\infty} a_n^*\left\{ S(y)\right\}^n.
\end{equation*}
\begin{proof} Let
\begin{equation*}
 \  \ \ a_n^*=\frac{a_n}{1-p_0} \ \  \mbox{ and }  \ \ b_n=p_1a_n^*
\end{equation*}
then
\begin{equation*}
\begin{aligned}
S_{zp}(y)&=p_0+\sum_{n=1}^{\infty} a_n\left\{ S(y)\right\}^n-\sum_{n=1}^{\infty} b_n\left\{ S(y)\right\}^n\\&=p_0+(1-p_0)\sum_{n=1}^{\infty} \frac{a_n}{1-p_0}\left\{ S(y)\right\}^n-\sum_{n=1}^{\infty} p_1a_n^*\left\{ S(y)\right\}^n \\&=p_0+(1-p_0)\sum_{n=1}^{\infty} a_n^*\left\{ S(y)\right\}^n-p_1\sum_{n=1}^{\infty} a_n^*\left\{ S(y)\right\}^n\\&=p_0+(1-p_0)S_p^*(y)-p_1S_p^*(y)\\&=p_0+(1-p_1-p_0)S_p^*(y).
\end{aligned}
\end{equation*}
\end{proof}
\end{theorem}

Note that $S_p^*(y)=P[Y>y|N\geq 1]$ is a proper survival function. Given a proper survival function, we have that
\begin{equation*}
    \lim_{y \to \infty} S_{zp}(y) = p_0,
\end{equation*}
while $S(0)_{zp} =(1-p_1)$, therefore $S_{zp}(y)$ accommodates both zero-adjustment (ZA) and the cure survival fraction. This model was presented by Louzada et al.\cite{louzada2018zero} to describe time-to-default in bank loan portfolios. The sub-density function is given by 
\begin{equation*}
\begin{aligned}
f(y)_{zp}=
\begin{cases}
p_1, & \text{if } y=0, \\
(1-p_1-p_0)f_p^*(y), & \text{if } y>0,
\end{cases}
\end{aligned}
\end{equation*}
where $f_p^*(y)$  is a proper probability density function (pdf) associated with the individuals susceptible to the event of interest.

It should be noted that (\ref{szp}) is necessary to achieve the standard mixture ZAC (\ref{zcuremodel}) model since the zero-adjusted standard distribution cannot be included directly in the unified cure model. 
Following Feller\cite{feller} and Rodrigues et al.\cite{rodrigues2009unification}, the probability generating function $A(u)=\theta+(1-\theta)u$ is used to achieve the standard cure model. Letting $p_0=\theta$ and $u=(1-p_1)S_p^*(y)$, and substituting in $A(\cdot)$, we have that
$
S_{zp}(y)=A((1-p_1)S_p^*(y))=p_0+(1-p_0)(1-p_1)S_p^*(y)=p_0+(1-p_1-p_0+p_0p_1)S_p^*(y),
$
\textcolor{black}{which represents a different parametrization when compared to (\ref{zcuremodel}). The main advantage of the formulation presented in (\ref{zcuremodel}) is that the parameters $p_0$ and $p_1$ are directly linked to the proportions of cure and zero-adjustment, respectively. This feature allows an easy interpretation of model estimates and facilitates the inclusion of covariates according to the research interests.  }

\begin{theorem} \label{teor5}
The  Promotion ZAC model is given by
\begin{equation} \label{promzicr}
S_{zp}(y)=\exp(-\theta F(y))-\frac{\exp(-\theta F(y)-\tau)-\exp(-\theta-\tau)}{1-\exp(-\theta)}
\end{equation} 
where $F(\cdot)$ is the cumulative distribution function of the baseline distribution of the random variable Y. The expression above is equivalent to
\begin{equation*}
S_{zp}(y)=\exp(-\theta)+ [1-\exp(-\tau)-\exp(-\theta)]S_p^*(y),
\end{equation*}
where
\begin{equation*}
S_p^*(y)=\frac{\exp(-\theta F(y))-\exp(-\theta)}{1-\exp(-\theta)}.
\end{equation*}

\end{theorem}

\begin{proof}
It is easy to show that 

\begin{equation}\label{sum1a}
\sum_{n=1}^{\infty} \frac{\theta^n \exp(-\theta)}{n!(1-\exp(-\theta))}\left\{ S(y)\right\}^n=\frac{\exp(-\theta F(y))-\exp(-\theta)}{1-\exp(-\theta)}.
\end{equation}

Recalling the equation

\begin{equation*}
\begin{aligned}
S_{zp}(y)&=a_0+\sum_{n=1}^{\infty} a_n\left\{ S(y)\right\}^n -\sum_{n=1}^{\infty}b_n\left\{ S(y)\right\}^n .
\end{aligned}
\end{equation*}

By considering that N follows a Poisson with parameter $\theta$, i.e., $a_n=\theta^n e^{-\theta}/n!$, then, if $b_n=\exp(-\tau)a_n^*$, where
\begin{equation*}
a_n^*=\frac{a_n}{(1-\exp(-\theta))}=\frac{\theta^n \exp(-\theta)}{n!(1-\exp(-\theta))}I_{\{1,2,\ldots\}}(n).
\end{equation*}

From (\ref{sum1a}), we have that
\begin{equation*}
\begin{aligned}
S_{zp}(y)&=a_0+(1-\exp(-\theta))\sum_{n=1}^{\infty} \frac{a_n}{(1-\exp(-\theta))}\left\{ S(y)\right\}^n -\sum_{n=1}^{\infty}b_n\left\{ S(y)\right\}^n \\& = \exp(-\theta)+(1-\exp(-\theta))\sum_{n=1}^{\infty} a_n^*\left\{ S(y)\right\}^n-\exp(-\tau)\sum_{n=1}^{\infty} a_n^*\left\{ S(y)\right\}^n\\& =\exp(-\theta)+(1-\exp(-\theta))S_p^*(y)-\exp(-\tau)S_p^*(y)\\&=\exp(-\theta)+(1-\exp(-\tau)-\exp(-\theta))S_p^*(y).
\end{aligned}
\end{equation*}

\end{proof}

The identifiability for the promotion cure model can be ensured (see Oliveira et al. \cite{de2017zero}) when covariates are included in both parameters related to the susceptible fraction and the cure fraction of individuals. On the other hand, mixture cure models have fewer restrictions (Li, et al. \cite{li2001identifiability}; Hanin and Huang \cite{hanin2014identifiability}) in practical applications so that mixture cure models are more popularly used and more intuitively interpretable than promotion cure models. 

In addition to the models (\ref{zcuremodel}) and (\ref{promzicr}), it is possible to assume more flexible probability distributions for the number of competing causes, $N$. Therefore, below, we present a more flexible particular case of the unified ZAC model, assuming a Negative Binomial (NB) distribution for the unobservable initial number of competing causes, $N$. In this scenario, the probability mass function of $N$ is given by


\begin{equation}\label{negrisl}
P[N=n;\theta,\eta]=\frac{\Gamma(\eta^{-1}+n)}{\Gamma(\eta^{-1})n!}\left(\frac{\eta\theta}{1+\eta\theta} \right)^n(1+\eta\theta)^{-1/\eta}, n=0,1,2,\ldots,
\end{equation}
where $\theta>0$, $\eta\geq-1$ and $\eta\theta+1>0$. 
Since $E(N)=\theta$ and $Var(N)=\theta+\eta\theta^2$, it is worth noting that positive values for $\eta$ correspond to an over-dispersion. In contrast, negative values correspond to an under-dispersion relative to the Poisson distribution.
The NB distribution has a variety of particular cases: Bernoulli ($\eta = -1$), Poisson ($\eta \rightarrow 0 $) and Geometric ($\eta = 1$).

The  ZAC model with NB causes, hereafter the ZAC-NB model, is given by
\begin{equation}\label{eqsugennb}
S_{zp}(y)=(1+\pi_0\eta \theta F(y))^{-1/\eta}-\frac{(1+\pi_0\eta \theta F(y))^{-1/\eta}-(1+\pi_0\eta \theta)^{-1/\eta}}{(1+\pi_1\eta \theta)^{-1/\eta}\left\{1-(1+\pi_0\eta \theta)^{-1/\eta}\right\}},
\end{equation}
where $0\leq\pi_1\leq1$ and $0\leq\pi_0\leq1-p_1$, where $p_1$ . Note that if $\pi_1=0$, then $(1+\pi_1\eta \theta)^{-1/\eta}=0$, the equation (\ref{eqsugennb}) reduces to 
$S(y)_{p}=(1+\pi_0\eta \theta F(y))^{-1/\eta}$, i.e., the  model proposed by Rodrigues et al.\cite{rodrigues2009unification}.  

In terms of mixture distribution, the ZAC-NB can be written as
\begin{equation} \label{BNZICRsur}
S_{zp}(y)=\left(\frac{1}{1+\eta\theta \pi_0} \right)^{\frac{1}{\eta}}+\left(1-\left(\frac{1}{1+\eta\theta \pi_0} \right)^{\frac{1}{\eta}} -\left(\frac{1}{1+\eta\theta \pi_1} \right)^{\frac{1}{\eta}}\right)S_p^*(y)
\end{equation}
where
\begin{equation}\label{eqlonb}
S_p^*(y)=\frac{(1+\pi_0\eta \theta F(y))^{-1/\eta}-(1+\pi_0\eta \theta)^{-1/\eta}}{1-(1+\pi_0\eta \theta)^{-1/\eta}},
\end{equation}
where $S_p^*(0)=1$ and $S_p^*(\infty)=0$, i.e. (\ref{eqlonb}) is a proper survival function. The density function for the non-cured population without ZA is given by
\begin{equation}\label{eqlonb2}
f_p^*(y)=\frac{\pi_0\theta f(y)(1+\pi_0\eta \theta F(y))^{-1/\eta-1}}{1-(1+\pi_0\eta \theta)^{-1/\eta}}.
\end{equation}

From $S_{zp}(y)$ given in (\ref{eqsugennb}), the density function of the ZAC-NB model is given by
\begin{equation*}
f_{zp}(y)=
\begin{cases}
\left(\frac{1}{1+\eta\theta \pi_1} \right)^{\frac{1}{\eta}}, & \text{if } y=0 \\
\left(1-\left(\frac{1}{1+\eta\theta \pi_0} \right)^{\frac{1}{\eta}} -\left(\frac{1}{1+\eta\theta \pi_1} \right)^{\frac{1}{\eta}}\right)f_p^*(y), & \text{if } y>0.
\end{cases}
\end{equation*}

 When $\eta = -1$ and $\eta \rightarrow 0$, the survival presented in (\ref{BNZICRsur}) reduces to the standard mixture ZAC and promotion ZAC models, respectively.  Additionally, when $\eta = 1$, we obtain the ZAC  model with Geometrically distributed causes (ZAC-Geo).

\subsection{Estimation}

In this section, we present the details about  the maximum likelihood estimators (MLE) for the parameters of the survival model (\ref{eqsugennb}) under censored data. Suppose that $T_i$ and $C_i$ are the lifetime and censoring time, respectively, for the $i$th  individual. Thus, the data is given by the \textcolor{black}{pairs 
 $\mathcal{D}=(y_i,\delta_i)$,} where $y_i=\min(T_i,C_i)$, $\delta_i=I(T_i\leq C_i)$, 
$i=1,2,\ldots m$ and $m$ is the sample size.
%
Then, assuming a random censoring \textcolor{black}{process, 
the }likelihood function related to the ZAC-NB is given by

{\small
\begin{equation*}\label{eqveroc3}
\begin{aligned}
 l(&\boldsymbol{\xi};\mathcal{D})=r\log\left(1-(1+\pi_0\eta \theta)^{-1/\eta}-(1+\pi_1\eta \theta)^{-1/\eta}\right)-r\log\left(1-(1+\pi_0\eta \theta)^{-1/\eta}\right)\\&+\sum_{i:y_i=0}\log\left(\frac{1}{1+\eta\theta \pi_1} \right)^{\frac{1}{\eta}} +r\log(\pi_0\theta)+\sum_{i:y_i>0}\log\left( f(y_i)(1+\pi_0\eta \theta F(y_i))^{-1/\eta} \right)\\& +\sum_{i:y_i>0}{(1-\delta_i)}\log \left\{ (1+\pi_0\eta \theta F(y_i))^{-1/\eta}-\frac{(1+\pi_0\eta \theta F(y_i))^{-1/\eta}-(1+\pi_0\eta \theta)^{-1/\eta}}{(1+\pi_1\eta \theta)^{1/\eta}\left\{1-(1+\pi_0\eta \theta)^{-1/\eta}\right\}} \right\} ,
\end{aligned}
\end{equation*}}
where $r=\sum_{i:y_i>0} \delta_i$ and $\boldsymbol{\xi}$ is a vector with $k$ parameters associated to $f_{zp}(y_i;\boldsymbol{\xi})$. The corresponding score function is $U(\boldsymbol{\xi})=\partial l(\boldsymbol{\xi};\mathcal{D})/\partial \xi_j, j=1,\ldots,k$. The maximum likelihood (ML) estimation for the parameter vector $\boldsymbol{\xi}$ can be implemented from the numerical maximization of the log-likelihood function (\ref{eqveroc3}).


To obtain the confidence intervals (CI) for the model parameters, it is reasonable to assume that the MLE is asymptotically distributed with the normal joint distribution given by
%
\begin{equation*} \boldsymbol{\hat{\xi}} \sim N_k[\boldsymbol{\xi},I^{-1}(\boldsymbol{\xi})] \mbox{ as } m \to \infty , \end{equation*}
where $I(\boldsymbol{\xi})$ is the Fisher information matrix, $k\times k$, and 
\begin{equation*}\label{fisherinf}
I_{jl}(\boldsymbol{\xi})=E\left[-\frac{\partial^2}{\partial \xi_j \partial \xi_l}l(\boldsymbol{\xi};\mathcal{D})^2\right],  j,l=1,\ldots,k.
\end{equation*}
As it is not trivial to obtain $I(\boldsymbol{\xi})$, we can consider the observed information matrix, $
H_{jl}(\boldsymbol{\xi})=-\frac{\partial^2}{\partial \xi_j \partial \xi_l}l(\boldsymbol{\xi};\mathcal{D})^2,\ j,l=1,\ldots,k.$
Therefore, we can obtain the approximated CI for a specific parameter, $\xi_j$, $j=1,\ldots,k$, assuming that the marginal distributions are given by
$\hat{\xi_j} \sim N[\xi_j,H^{-1}_{jj}(\boldsymbol{\xi})] \mbox{ for } m \to \infty$.

\section{Simulation study}
\subsection {Evaluating performance of point and interval estimates }

We performed a  Monte Carlo simulation study to assess the MLE and asymptotic CI performance for finite samples. Therefore, to measure the performance of point estimates, we consider bias and root of mean square errors (RMSE), which are obtained as follows 
\begin{equation*}
\f{Bias}_j=\frac{1}{B}\sum_{b=1}^{B}(\hat\theta_{j,b}-\theta_j) \ ,  \quad \textcolor{black}{ \f{RMSE}_j=\sqrt{\frac{1}{B}\sum_{b=1}^{B}(\hat\theta_{j,b}-\theta_j)^2},}
\end{equation*} 
$\f{for} j=1,\ldots,k$ where $B=10,000$ is the number of simulated samples and $\hat\theta_{j,b}$ is the estimate of $\theta_j$ in the $b$th sample. In this scenario, the Bias and the RMSEs should be close to zero. 
To measure the performance of asymptotic CIs, we present the coverage probability (CP), which should be close to the nominal confidence level of 95\%. 


The samples were simulated  assuming the competing causes are i.i.d. with NB distribution given in (\ref{negrisl}), with $\eta$ known: $\eta=-1$ (standard mixture ZAC), $\eta\rightarrow 0$ (promotion ZAC), $\eta=0.5$, $\eta=1$ (NB-Geo) and $\eta=2$. Furthermore, we assume that the unobservable baseline distribution $F(y)$ follows a lognormal distribution where $\mu$ and $\sigma$ are the location and scale parameters.  The chosen values of the simulation parameters were:

\begin{itemize}
    \item Standard mixture ZAC: $(\mu,\sigma,p_0,p_1)=(2,1,0.3,0.1)$ 
    \item Promotion ZAC:  $(\mu,\sigma,\tau,\theta)=(2,1,1.2,2.3)$ 
    \item ZAC-NB with $\eta=0.5$ : $(\mu,\sigma,\theta\pi_0,\theta\pi_1)=(4,1.5,3,3)$
    \item ZAC-Geo $(\mu,\sigma,\theta\pi_0,\theta\pi_1)=(5,1.5,5,5)$
    \item ZAC-NB with $\eta=2$: $(\mu,\sigma,\theta\pi_0,\theta\pi_1)=(5,1.5,19,19).$
\end{itemize}  

\textcolor{black}{The values were above were chosen to obtain different levels of censoring where the mean proportion of right-censored data are 0.351 ($\eta=-1$), 0.372 ($\eta\rightarrow 0$), 0.239 ($\eta=0.5$), 0,260 ($\eta=1$) and 0,194($\eta=0.2$).} The methodology used to simulate censored data is based on \cite{ramos2024sampling}. The maximum likelihood estimates were achieved by numerical optimization using the \textit{optim} function available in the R software.

Table \ref{tab:simu1} presents the Bias, MSEs, and CP for the obtained estimates for different sample sizes. The results indicate that Bias and RMSE are closer to zero, and the empirical coverage probabilities are closer to the nominal coverage level as the sample size increases. These results are expected if the underlying estimation scheme works correctly to produce consistent and asymptotically normal estimates.  The results
were similar for other choices of $\boldsymbol{\xi}$. Regarding the stability of the estimates, we found that, on average, in 1.7\% of the samples, the mle could not find the estimates due to the lack of convergence. This result indicates no problem of instability in the estimates, which is one of the main consequences of the lack of identifiability in models.

\textcolor{black}{It is worth mentioning that we observed the RMSE and bias for the parameters \(p_0\) and \(p_1\) are notably higher when using the ZAC-NB model in scenarios with small sample sizes (\(n=50\) or \(n=100\)) and \(\eta = 2\). This behavior may be attributed to the complexity of the Negative Binomial distribution, which introduces additional variability in the estimation process. Furthermore, the challenge is compounded by the simultaneous estimation of four parameters and the presence of right-censored data, which, due to the small sample size available for estimation, contributes to larger bias and variability in the parameter estimates, especially for smaller samples. However, as the sample size increases, the bias and RMSE decrease substantially, consistent with the asymptotic properties of the maximum likelihood estimators. This improvement is reflected in the coverage probabilities, which approach their nominal levels for larger sample sizes. These findings highlight the sensitivity of the ZAC-NB model to small sample sizes.
}

\begin{landscape}

 %
   
\begin{table}[htbp]
  \centering
  \caption{Bias, root of mean
square error (RMSE) and coverage probability (CP) of 95\% 
CIs for the parameters in the ZAC model with $\eta=-1$, $\eta\rightarrow0$, $\eta=0.5$, $\eta=1$ and $\eta=2$    and lognormal baseline distribution.}
  \resizebox{21cm}{!}{%
    \begin{tabular}{lccccccccccccccccccccc} \hline
    &      & \multicolumn{4}{c}{\textbf{Mixture} ($\eta=-1$) }   & \multicolumn{4}{c}{\textbf{Promotion} ($\eta\rightarrow0$) } & \multicolumn{4}{c}{\textbf{ZAC-NB} ($\eta=0.5$) }                & \multicolumn{4}{c}{  \textcolor{black}{\textbf{ZAC-Geo} ($\eta=1$)} }    & \multicolumn{4}{c}{ \textcolor{black}{\textbf{ZAC-NB} ($\eta=2$)} }   \\ 
    n     & & \multicolumn{1}{c}{Bias} & \multicolumn{1}{c}{RMSE} & \multicolumn{1}{c}{CP} &       & \multicolumn{1}{c}{ Bias } & \multicolumn{1}{c}{RMSE} & \multicolumn{1}{c}{ CP} &       &       & \multicolumn{1}{c}{Bias} & \multicolumn{1}{c}{RMSE} & \multicolumn{1}{c}{CP} &       & \multicolumn{1}{l}{Bias} & \multicolumn{1}{c}{RMSE} & \multicolumn{1}{c}{CP} &       & \multicolumn{1}{c}{Bias} & \multicolumn{1}{c}{RMSE} & \multicolumn{1}{c}{CP} \\ \hline
    \multicolumn{1}{c}{50} &  $\mu$  & 0.0100 & 0.2276 & 0.941 &  $\mu$  & 0.0137 & 0.2685 & 0.941 &  $\mu$  & & 0.1514 & 0.9754 & 0.899 &  $\mu$  & 0.0963 & 0.9310 & 0.896 &  $\mu$  & -0.0898 & 0.9066 & 0.897 \\
                       &  $\sigma$  & -0.0243 & 0.1705 & 0.902 &  $\sigma$  & -0.0159 & 0.1711 & 0.914 &  $\sigma$  &       & -0.0241 & 0.2272 & 0.917 &  $\sigma$  & -0.0302 & 0.2179 & 0.931 &  $\sigma$  & -0.0560 & 0.2094 & 0.928 \\
                       &  $p_0$  & 0.0014 & 0.0418 & 0.886 &  $\tau$  & 0.148 & 0.8659 & 0.971 &  $\tau$  &       & 0.0007 & 0.0737 & 0.907 &  $\tau$  & 1.6603 & 6.2415 & 0.899 &  $\tau$  & 4.1197 & 17.7588 & 0.878 \\
                       &  $p_1$  & -0.0044 & 0.0770 & 0.942 &  $\theta$  & 0.0394 & 0.2735 & 0.957 &  $\theta$  &       & -0.0003 & 0.0497 & 0.963 &  $\theta$  & 0.8074 & 3.2080 & 0.937 &  $\theta$  & 4.5027 & 17.2923 & 0.917 \\ \hline
    \multicolumn{1}{r}{150} &  $\mu$  & 0.0057 & 0.1277 & 0.948 &  $\mu$  & 0.0026 & 0.1466 & 0.944 &  $\mu$  &       & 0.0532 & 0.3742 & 0.944 &  $\mu$  & 0.0380 & 0.4326 & 0.942 &  $\mu$  & 0.0272 & 0.5348 & 0.936 \\
                       &  $\sigma$  & -0.0060 & 0.0960 & 0.934 &  $\sigma$  & -0.007 & 0.0959 & 0.935 &  $\sigma$  &       & -0.0007 & 0.1083 & 0.953 &  $\sigma$  & -0.0058 & 0.1091 & 0.949 &  $\sigma$  & -0.0079 & 0.1104 & 0.946 \\
                       &  $p_0$  & -0.0005 & 0.0247 & 0.921 &  $\tau$  & 0.0346 & 0.2645 & 0.958 &  $\tau$  &       & -0.0047 & 0.0366 & 0.957 &  $\tau$  & 0.4094 & 2.0255 & 0.934 &  $\tau$  & 2.5987 & 10.5494 & 0.913 \\
                       &  $p_1$  & -0.0009 & 0.0436 & 0.949 &  $\theta$  & 0.011 & 0.1429 & 0.950 &  $\theta$  &       & -0.0002 & 0.0303 & 0.943 &  $\theta$  & 0.2214 & 1.2578 & 0.945 &  $\theta$  & 2.2886 & 9.3500 & 0.928 \\ \hline
    \multicolumn{1}{r}{300} &  $\mu$  & 0.0004 & 0.0890 & 0.948 &  $\mu$  & 0.0002 & 0.1021 & 0.946 &  $\mu$  &       & 0.0048 & 0.2078 & 0.946 &  $\mu$  & 0.0172 & 0.2589 & 0.951 &  $\mu$  & 0.0149 & 0.3503 & 0.949 \\
                       &  $\sigma$  & -0.0043 & 0.0669 & 0.947 &  $\sigma$  & -0.0029 & 0.0666 & 0.947 &  $\sigma$  &       & -0.0053 & 0.0671 & 0.953 &  $\sigma$  & -0.0021 & 0.0706 & 0.949 &  $\sigma$  & -0.0029 & 0.0725 & 0.952 \\
                       &  $p_0$  & -0.0001 & 0.0174 & 0.936 &  $\tau$  & 0.0132 & 0.1794 & 0.953 &  $\tau$  &       & 0.0006 & 0.0237 & 0.953 &  $\tau$  & 0.1616 & 0.9917 & 0.948 &  $\tau$  & 1.3043 & 6.8629 & 0.938 \\
                       &  $p_1$  & -0.0001 & 0.0308 & 0.949 &  $\theta$  & 0.0053 & 0.1000   & 0.951 &  $\theta$  &       & -0.0003 & 0.0210 & 0.951 &  $\theta$  & 0.1084 & 0.8230 & 0.953 &  $\theta$  & 1.1499 & 6.0057 & 0.944 \\ \hline
    \multicolumn{1}{r}{500} &  $\mu$  & 0.0012 & 0.0686 & 0.949 &  $\mu$  & 0.0021 & 0.0785 & 0.953 &  $\mu$  &       & -0.0024 & 0.1458 & 0.953 &  $\mu$  & 0.0034 & 0.1804 & 0.954 &  $\mu$  & 0.0132 & 0.2585 & 0.952 \\
                       &  $\sigma$  & -0.0016 & 0.0525 & 0.944 &  $\sigma$  & -0.0009 & 0.0516 & 0.948 &  $\sigma$  &       & -0.0042 & 0.0490 & 0.949 &  $\sigma$  & -0.0027 & 0.0515 & 0.951 &  $\sigma$  & -0.0006 & 0.0539 & 0.951 \\
                       &  $p_0$  & 0.0000 & 0.0134 & 0.946 &  $\tau$  & 0.0081 & 0.1386 & 0.949 &  $\tau$  &       & -0.0001 & 0.0181 & 0.955 &  $\tau$  & 0.0769 & 0.6824 & 0.954 &  $\tau$  & 0.7510 & 4.7823 & 0.943 \\
                       &  $p_1$  & -0.0006 & 0.0236 & 0.950 &  $\theta$  & 0.0056 & 0.0766 & 0.952 &  $\theta$  &       & 0.0002 & 0.0157 & 0.962 &  $\theta$  & 0.0636 & 0.6211 & 0.951 &  $\theta$  & 0.6342 & 4.2879 & 0.946 \\ \hline
    \end{tabular}%
    }%
  \label{tab:simu1}%
\end{table}%

\end{landscape}

\newpage
\subsection {Evaluating mis-specification }

As highlighted above, the main advantage of the unified ZAC model is the flexibility for adjusting different distributions related to competing causes. In this context, it is crucial to assess whether it is possible to discriminate among models, correctly identifying the true data distribution. 
Such discrimination can be made by considering a model mis-specification study addressed here. 

Four scenarios are considered:

\begin{itemize}
\item Scenario I: Generate samples from ZAC-NB model with $\eta=0.5$ and Log-Normal baseline ($\mu=4, \sigma=1.5, p_0= 0.16, p_1= 0.16  $ )  and fit the data with estimators based on ZAC-NB model with $\eta=-1$ and the same baseline distribution;
  \item Scenario II: Generate samples from ZAC-NB model with $\eta=2$ and Log-Normal baseline ($\mu=4 , \sigma=1.5, p_0= 0.16, p_1= 0.16  $ ) and fit the data with estimators based on ZAC-NB model with $\eta \rightarrow$ 0 and the same baseline distribution;
 \item Scenario III: Generate samples from ZAC-NB model with $\eta=-1$ and Gompertz baseline ($\alpha=0.1, \beta=1.5, p_0= 0.16, p_1= 0.16$)  and fit the data with estimators based on ZAC-NB model with a Log-Normal baseline and the same value of $\eta$;
 \item Scenario IV: Generate samples from ZAC-NB model with $\eta=0.5$ and Gompertz  baseline ($\alpha=0.1, \beta=1.5, p_0= 0.16, p_1= 0.16$) and fit the data with estimators based on ZAC-NB model with a Log-Normal baseline and the same value of $\eta$;
  \item Scenario I.a: Generate samples from ZAC-NB model with $\eta=0.5$ and Log-Normal baseline ($\mu=4, \sigma=1.5, p_0= 0.16, p_1= 0.16$), as in the scenario I, but fitting the data based on promotion ZAC model ($\eta \rightarrow 0$) and the same baseline distribution;
   \item Scenario I.b: Generate samples from ZAC-NB model with $\eta=0.5$ and Log-Normal baseline ($\mu=4, \sigma=1.5, p_0= 0.25, p_1= 0.25  $ )  and fit the data with estimators based on ZAC-NB model with $\eta = -1 $ and the same baseline distribution;
\end{itemize}

In the scenarios III and IV, we consider the Gompertz distribution, with pdf given by $  F(t | \alpha, \beta) = 1 - \exp\{-\frac{\beta}{\alpha} [\exp(\alpha t) - 1]\}  $. Scenarios I.a and I.b are variations of scenario I, so that in I.a, the $\eta$ values are closer and in I.b, the long-term proportions and zero-adjustment are higher.
For each scenario, 1,000 samples were generated, and the Akaike criterion (AIC) \cite{akaike1974new} was calculated for the fitted model. The observed AIC for the misfitted was compared with the AIC for the model, considering the correct data distribution. Table \ref{tab:tab2}  presents the proportion of samples that the AIC for the misfitted model is greater than the AIC for the model considering the correct data distribution, i.e., where the AIC correctly discriminate the distribution.
\begin{table}[!h]
  \centering
  \caption{{Proportion of samples which the AIC for the mis-fitted model is greater than the AIC for the correct model in each simulation scenario}}
    \begin{tabular}{lrrrrrr} \hline
       & \multicolumn{4}{c}{Scenarios} \\ 
      $n$       & \multicolumn{1}{l}{I}    &\multicolumn{1}{c}{II} & \multicolumn{1}{c}{III} & \multicolumn{1}{c}{IV} & \multicolumn{1}{c}{\textcolor{black}{I.a}}& \multicolumn{1}{c}{\textcolor{black}{I.b}} \\ \hline
    50  & 0.497   & 0.621 & 0.500   & 0.399& 0.467&0.469 \\
    150 & 0.664   & 0.767 & 0.769 & 0.610& 0.532& 0.584\\
    300 & 0.790   & 0.852 & 0.928 & 0.828& 0.624& 0.657 \\
    500 & 0.838  & 0.947  & 0.978 & 0.927 & 0.630& 0.740\\ \hline
    \end{tabular}%
  \label{tab:tab2}%
\end{table}%

The AIC presents better discrimination for all scenarios as the sample size increases. Most cases present discrimination above 60\% when $n=150$ and above 70\% when $n=300$. \textcolor{black}{In scenarios Ia and Ib, the compared models exhibit very similar density functions, making discrimination between them inherently challenging, particularly when the sample size is small. This issue arises not from limitations in the proposed approach but from the intrinsic similarity between the models in these cases. Adequate discrimination in such scenarios requires larger sample sizes to better capture subtle differences between the models. To address this limitation, future work will explore alternative estimation methods, such as Bayesian inference, and the incorporation of covariates, which are expected to enhance the discriminatory power of the method. Overall, these results indicate that discrimination improves as the values of $\eta$ are distant from each other and the proportion of cure or zero-adjustment decreases.}

\section{Application: A Sub-Saharan African obstetric data}

This section presents the results of applying the proposed model in a sub-Saharan African obstetric study, developed by WHO as part of the BOLD project \cite{souza2015development}. Because this project aims to improve the quality of intrapartum care and generate evidence-based tools, it is essential to consider all events involved in health care management to ensure reliability.
Concerning intrapartum care and, notably, the labor time (i.e., the time between hospital admission and vaginal birth), three different women groups can be observed: a proportion of women with fetal death at admission (representing the zero adjustments), a proportion who undergo to C-section, and the third group, which presents a normal progression, with vaginal delivery. Women who underwent C-section have their natural delivery time reduced and, therefore, their labour times were considered censored. On the other hand, all women with vaginal deliveries have full-time records, i.e., non-censored. The critical point is that for the women who arrive at the hospital already having a stillbirth, \textcolor{black}{the time is generally not registered}, leading to the necessity to consider that the time is equal to zero (T=0). 
 Usual survival models do not allow us to consider lifetimes equal to zero. Therefore, for the model proposed in Section \ref{sec2}, it is possible to account for this feature and consider the three women groups in the process of survival estimation concomitantly.
 
 The present dataset contains labor times of $7,062$ pregnant women, selected according to clinical characteristics of interest. In this sample, 319 women presented labour time equal to zero, and 2796 presented censored times.

\begin{table}[!h]
  \centering
  \caption{  Akaike criterion for the standard mixture ZAC, promotion ZAC, ZAC-NB, and ZAC-Geo models with different baseline distributions: exponential, log-normal, Gompertz and inverse Gaussian.}
    \begin{tabular}{lcccc} \toprule
    \textbf{Baseline } & \multicolumn{1}{c}{S. Mixture} & \multicolumn{1}{c}{Promotion} & \multicolumn{1}{c}{NB ($\eta=0.5$)} & \multicolumn{1}{c}{Geo ($\eta=1$)} \\ \hline
    Exponential & 31716,07 & 31737,41 & 31920,84 & 32285,13 \\
    Log-normal & \textbf{31034,95} & \textbf{31035,61} & \textbf{31033,08} & \textbf{31029,67} \\
    Gompertz & 31736,83 & 31705,85 & 31732,49 & 31724,32 \\
    Inverse Gaussian & 31472,52 & 32027,4 & 32397,49 & 32767,6 \\ \bottomrule
    \end{tabular}%
  \label{tab:aic}%
\end{table}%

To fit the proposed model to the data, it is important to evaluate which $\eta$ value is more appropriate and which base distribution best fits the data. In this way, we fitted models with four different $\eta$ values, including the main particular cases of the unified model: standard mixture ZAC model ($\eta=-1$), promotion ZAC model($\eta\rightarrow 0$), ZAC-NB ($\eta = 0.5$) and ZAC-Geo($\eta=1$). Additionally, for each possible value of $\eta$, we fitted four different baseline distributions: exponential, log-normal, Gompertz, and inverse Gaussian. Thus, 16 possible models were evaluated and, in Table \ref{tab:aic} presents the Akaike criterion (AIC) values for each model. It is possible to observe that, in all adjusted models, the log-normal distribution was the one with the lowest AIC values in all the particular cases of the unified model.  This result corroborates with what is observed in the literature. The log-normal distribution is known for its utility in the childbirth context \cite{zhang2010contemporary}.
 In this sense, considering the log-normal as the baseline distribution, Figure \ref{fig:fit} presents fitted survival for each model compared with Kaplan-Meier estimates, and Table \ref{tab:fit} presents the MLE and standard deviations for the main particular cases of the unified model.

\begin{center}
\begin{table*}[!h]%
{\scriptsize
\centering
  \caption{Maximum likelihood estimates (MLE) and Standard deviations (SD) for the standard mixture ZAC, promotion ZAC, ZAC-NB, and ZAC-Geo models {with log-normal baseline distribution.} }
       \begin{tabular}{lccrccrccccc} \toprule
& \multicolumn{2}{c}{S. Mixture } &       & \multicolumn{2}{c}{Promotion} &       & \multicolumn{2}{c}{NB ($\eta=0.5$)} &       & \multicolumn{2}{c}{Geo ($\eta=1$)} \\ 
  & \multicolumn{1}{r}{MLE} & \multicolumn{1}{r}{SD} & & \multicolumn{1}{r}{MLE} & \multicolumn{1}{r}{SD} &  & \multicolumn{1}{r}{MLE} & \multicolumn{1}{r}{SD} & \multicolumn{1}{l}{} & \multicolumn{1}{r}{MLE} & \multicolumn{1}{r}{SD} \\
   \midrule
    $\mu$ & \multicolumn{1}{r}{2.167} & \multicolumn{1}{r}{0.024} & \multicolumn{1}{l}{$\mu$} & \multicolumn{1}{r}{3.200} & \multicolumn{1}{r}{0.088} & \multicolumn{1}{l}{$\mu$} & \multicolumn{1}{r}{4.151} & \multicolumn{1}{r}{0.185} & \multicolumn{1}{l}{$\mu$} & \multicolumn{1}{r}{5.8163} & \multicolumn{1}{r}{0.4375} \\
    $\sigma$ & \multicolumn{1}{r}{1.035} & \multicolumn{1}{r}{0.016} & \multicolumn{1}{l}{$\sigma$} & \multicolumn{1}{r}{1.259} & \multicolumn{1}{r}{0.025} & \multicolumn{1}{l}{$\sigma$} & \multicolumn{1}{r}{1.430} & \multicolumn{1}{r}{0.034} & \multicolumn{1}{l}{$\sigma$} & \multicolumn{1}{r}{1.6848} & \multicolumn{1}{r}{0.0490} \\
    $p_1 $ & \multicolumn{1}{r}{0.045} & \multicolumn{1}{r}{0.003} & \multicolumn{1}{l}{$\tau $} & \multicolumn{1}{r}{3.092} & \multicolumn{1}{r}{0.058} & \multicolumn{1}{l}{$\theta \pi_1$} & \multicolumn{1}{r}{7.383} & \multicolumn{1}{r}{0.058} & \multicolumn{1}{l}{$\theta \pi_1$} & \multicolumn{1}{r}{21.0093} & \multicolumn{1}{r}{0.0575} \\
    $p_0 $ & \multicolumn{1}{r}{0.038} & \multicolumn{1}{r}{0.011} & \multicolumn{1}{l}{$\theta$} & \multicolumn{1}{r}{3.248} & \multicolumn{1}{r}{0.221} & \multicolumn{1}{l}{$\theta \pi_0$} & \multicolumn{1}{r}{9.726} & \multicolumn{1}{r}{0.259} & \multicolumn{1}{l}{$\theta \pi_0$} & \multicolumn{1}{r}{64.4428} & \multicolumn{1}{r}{0.3888}\\ \bottomrule
    \end{tabular}%
  \label{tab:fit}}%
\end{table*}%
\end{center}

\begin{figure}[!h]
\centering
\includegraphics[width=12cm]{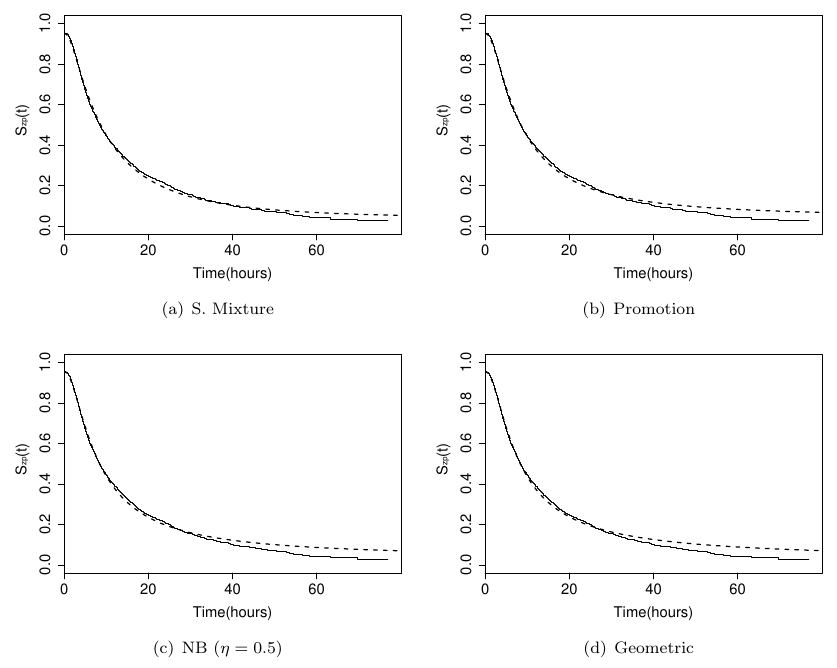}
\caption{Survival curves for the standard mixture ZAC(a), promotion ZAC(b), ZAC-NB (c) and ZAC-Geo (d) models with log-normal baseline distribution. The dashed lines correspond to the fitted model while the solid lines correspond to Kaplan-Meier estimates.}
\label{fig:fit}
\end{figure}

\textcolor{black}{
In Table \ref{delta_aic}, the delta AIC values \cite{anderson2004model} are presented, which can be used to evaluate the relative differences between models. These values measure how well each model fits the data, allowing for the comparison of competing risk approaches.}

\begin{table}[h!]
\centering
\caption{
\textcolor{black}{Delta AIC values and their interpretation for each comparison. Model 1 corresponds to the S. Mixture, Model 2 to the Promotion model, Model 3 to NB ($\eta = 0.5$), and Model 4 to Geo ($\eta = 1$).}}
\begin{tabular}{c|c|c}
\hline
Comparison & Delta AIC & Interpretation \\ \hline
1 vs 2     & 0.66      & Insignificant difference \\ \hline
1 vs 3     & 1.87      & Insignificant difference \\ \hline
1 vs 4     & 5.28      & Moderate evidence \\ \hline
2 vs 3     & 2.53      & Marginal evidence \\ \hline
2 vs 4     & 5.94      & Moderate evidence \\ \hline
3 vs 4     & 3.41      & Moderate evidence \\ \hline
\end{tabular}
\label{delta_aic}
\end{table}

\textcolor{black}{
The analysis of the delta AIC values reveals moderate evidence of differences between the models, particularly when comparing certain configurations, such as ``1 vs 4,'' ``2 vs 4,'' and ``3 vs 4,'' where the delta AIC values fall in the range of 3 to 7. These results suggest that while some models may perform similarly in explaining the data, there are notable distinctions in their fit, highlighting the benefits of considering different competing risk approaches. By examining these variations, researchers can better capture the nuances in the underlying data, leading to more robust and tailored insights. The moderate evidence observed underscores the importance of evaluating multiple models in such analyses, as it allows for a more comprehensive understanding of the competing risks and ensures that the most appropriate model is selected for the study context.}

According to AIC values, the ZAC-Geo presents the best fit to the data. It is worth mentioning that both ZAC-Geo and ZAC-NB are new ZAC models obtained from our proposed approach. The estimates of zero-adjustment and cure from the fitted ZAC-Geo model can be obtained from equation (\ref{BNZICRsur}), considering $\eta = 1$. Then, the zero-adjustment percentage is equal to  $4.544\%$ $\left( \left(\frac{1}{1+\eta\theta \pi_1} \right)^{\frac{1}{\eta}}  = (1 + 21.0093)^{-1}\right)$ and the percentage of cure is 1.528\% $\left( \left(\frac{1}{1+\eta\theta \pi_0} \right)^{\frac{1}{\eta}} = (1 + 64.4428)^{-1}\right)$.
Based on the survival estimates presented in Table \ref{tab:tab4}, at admission 95.46\% of women still have the possibility of vaginal delivery. This value decreases to 61.82\% in 6 hours, 37.53\% in 12 hours, and 26.16\% in 18 hours. These values could support the facilities' management to organize care for \textcolor{black}{pregnant women}.

\begin{center}
  \begin{table*}[h]
  \centering
   \caption{Estimated survival in the zero-adjusted cure survival model with Geometric causes {with log-normal baseline distribution}. }
    \begin{tabular}{lrrrr} \toprule
    Time  & \multicolumn{1}{l}{Admission} & \multicolumn{1}{l}{6 hours} & \multicolumn{1}{l}{12 hours} & \multicolumn{1}{l}{18 hours} \\ \midrule
    $S_{zp}$(y)  & 95.46\% & 61.82\% & 37.53\% & 26.16\% \\ \bottomrule
    \end{tabular}%
  \label{tab:tab4}%
\end{table*}%
\end{center}

\section{Discussion}

This paper presents a unified approach for survival models, considering two essential features: zero-adjustment and cure fraction. The main advantage of this methodology is that it ensures excellent flexibility for modeling. Using the proposed unified ZAC model, a researcher can consider different probability distributions for susceptible individuals' lifetimes. Different distributions can also be considered for the competing causes associated with the occurrence of the event of interest. Our model has particular cases other recently proposed models such as the unified cure models \cite{rodrigues2009unification}, and the mixture and promotion ZAC models \cite{louzada2018zero,de2017zero}. The standard mixture cure model by Berkson and Gage\cite{berkson1952survival} and the promotion cure model \cite{chen1999new} are also particular cases of the proposed model. On the other hand, the ZAC-NB and ZAC-Geo are new in literature, and the obtained models assuming the exponential, log-normal, Gompertz, and Inverse Gaussian distribution as a baseline.
More importantly, if the researchers do not know the random mechanism that generates the data, they can evaluate the goodness of fit for several models, as well as propose, using our approach, new ZAC models from other different discrete distributions in the competitive risk latent variable and other baseline distributions. From there, infer which particular case is most suitable for the data.

In the application section, we exemplify the proposed model's use with sub-Saharan African labor data, considering the susceptible group's log-normal distribution, in agreement with the results of \cite{vahratian2006methodological,zhang2010contemporary}. The results indicate that the ZAC-Geo model excels at others, showing an advantage in modeling by considering more flexible models. It is also possible to infer that there is ZAC competing causes related to the primary outcome (vaginal delivery), which may indicate different causes for medical decision-making during labor.  
We observed that about 5\% of the pregnant women arrive at the hospital already presenting fetal death. Therefore, we estimate that 95\% of women can go through a vaginal birth with a live baby when they arrive at the hospital. After 12 hours of admission, for example, 37\% of the women remain \textcolor{black}{who have not yet} had vaginal delivery. These results can be useful for hospital managers to know the dynamics of entry and exit of pregnant women in hospitals.

An issue with more flexible models, such as the unified model proposed here, is identifiability.  One of the main implications of the lack of identifiability is the instability of estimates \cite{hanin2014identifiability}. The simulation results presented here do not give any indication of instability in the estimates. Besides, they indicate that the AIC can be used to compare models, mainly for larger sample sizes and values of $\eta$ distant from particular cases.  However, since identifiability depends on the baseline distribution, this issue must be investigated for the particular constructed model. 

\textcolor{black}{A common approach to overcome this problem is to include covariates. By incorporating covariates into the unified framework can be achieved by linking parameters of interest, such as the cure fraction (\(p_1\)), the proportion of zero-inflated events (\(p_0\)), and baseline hazard parameters (\(\phi\)), to covariates through systematic components. This approach involves defining link functions that map covariates to these parameters while ensuring they remain within valid ranges. For example, the proportion of zero-inflated events (\(p_0\)) and the cure fraction (\(p_1\)) can be modeled using a multinomial logit link, while the baseline hazard parameters, such as the shape and scale in a Weibull distribution, can be linked to covariates using a log link function.}

\textcolor{black}{Specifically, consider the following structure:
\[
H(p_{0i}, p_{1i}) = \left(\log\left(\frac{p_{0i}}{1 - p_{0i} - p_{1i}}\right), \log\left(\frac{p_{1i}}{1 - p_{0i} - p_{1i}}\right)\right),
\]
where the linear predictors \(\xi_{0i}\) and \(\xi_{1i}\) are defined as \(\xi_{0i} = x_{1i}^\top \beta_1\) and \(\xi_{1i} = x_{2i}^\top \beta_2\), respectively. Similarly, for baseline parameters \(\phi = \{\alpha, \theta\}\), assuming a Weibull baseline distribution, covariates can be incorporated using:
\[
g_1(\alpha_i) = \log(\alpha_i), \quad g_2(\theta_i) = \log(\theta_i).\]}

\textcolor{black}{This formulation not only allows the analyst to incorporate subject-specific factors into the model but also maintains the flexibility of the unified framework by enabling the use of various competing cause distributions (e.g., Binomial, Negative Binomial). For instance, in a clinical context, patient demographics, comorbidities, or treatment regimens could directly influence the cure fraction or the likelihood of experiencing instantaneous failures (zero inflation).}

In general, we have shown that our model is a useful tool for adjusting survival data in the presence of zeros, allowing the evaluation of different particular cases with great flexibility. It is essential to highlight that this model can be applied in several other practical situations that need to accommodate zero adjusted lifetimes. Thus, we believe that the model has a vast application potential beyond the labor example presented here.

\section*{Statements of ethical approval}

The BOLD cohort obtained ethical approval from the following committees: World Health Organization Ethical Review Committee, Makerere University School of Health Sciences Research and Ethics Committee, Uganda,  University of Ibadan/University College Hospital Ethics Committee, Federal Capital Territory Health Research Ethics Committee, Nigeria, and Ondo State Government Ministry of Health Research Ethics Review Committee, Nigeria.

\section*{Declaration of Competing Interest}

No potential conflict of interest was reported by the author(s).

\section*{Acknowledgements}

\textcolor{black}{The authors thank the associate editor and two anonymous referees for their constructive comments and suggestions.} The dataset considered in this paper was obtained as one of the BOLD project stages, with researchers and health professionals' collaboration.  Therefore, we acknowledge the contributions of the country leaders.
We would also like to thank BOLD coordinators Jo\~ao Paulo Souza and Olufemi T. Oladapo. \textcolor{black}{The code used in this study is available at: https://github.com/ramospedroluiz/unifzacr/}

\section*{Funding}

The BOLD project was funded by the Bill \& Melinda Gates Foundation, United States Agency for International Development, and the UNDP-UNFPA-UNICEF-WHO-World Bank Special Program of Research, Development, and Research Training in Human Reproduction (HRP). The researchers were partially supported by CNPq, FAPESP, CAPES, FAEPA of Brazil, and the Public Health Graduate Program from Ribeir\~ao Preto Medical School, University of S\~ao Paulo, Brazil. The research of Francisco Louzada is supported by FAPESP (Grant number: 2013/07375-0) and CNPq (Grant number: 308849/2021-3).

\bibliographystyle{plain}

\end{document}